\documentclass[12pt]{article}

\usepackage[utf8]{inputenc}
\usepackage{hyperref}
\usepackage{amsmath, amsfonts, amssymb, amsthm, cases, caption, csquotes, enumitem, graphicx, mathrsfs, mathtools, tikz, tikz-cd,tkz-euclide}
\usepackage[style=alphabetic]{biblatex}

\theoremstyle{definition}

\newtheorem{definition}{Definition}[section]
\newtheorem{theorem}{Theorem}[section]
\newtheorem{proposition}{Proposition}[section]

\newtheorem{lemma}{Lemma}[section]

\newcommand{\Addresses}{{
		\bigskip
		\footnotesize
		\textsc{Department of Mathematics, University of Toronto, 40 St. George St., Toronto, ON, Canada}\par\nopagebreak
		\textit{E-mail address: } \texttt{petr.kosenko@mail.utoronto.ca}
}}

\title{Fundamental inequality for hyperbolic Coxeter and Fuchsian groups equipped with geometric distances}
\author{Petr Kosenko}

\usepackage[left=2cm,right=2cm,top=2cm,bottom=2cm,bindingoffset=0cm]{geometry}

\begin{document}
	\maketitle
	\begin{abstract}
		We prove that the hitting measure is not equivalent to the Lebesgue measure for a large class of nearest-neighbour random walks on hyperbolic reflection groups and Fuchsian groups.
	\end{abstract}
	\section{Introduction}
	Fix a hyperbolic regular polygon $\Delta_{n,m} \subset \mathbb{H}^2$ with $n$ sides and interior angles equal to $\frac{2\pi}{m}$. If $m \ge 4$ is even, then the group generated by reflections $r_i$ with respect to the sides of $\Delta_{n,m}$ is called the \textbf{hyperbolic Coxeter group}
	\[
	\Gamma_{n,m} := \left< r_1, \dots, r_n \ | \ r_i^2 = (r_i r_{i+1})^{m/2} = e \right>.
	\]
	Therefore, $\Gamma_{n,m}$ is equipped with a natural geometric (that is, isometric, cocompact, and properly discontinuous) action on $\mathbb{H}^2$ (see \cite[Theorem 6.4.3]{book:62172}), which makes it a word-hyperbolic group.
	
	If a hyperbolic group $\Gamma$ is equipped with a geometric action on $\mathbb{H}^2$, we can fix a point $x_0 \in \mathbb{H}^2$ with $\text{Stab}_{x_0} = \{e\}$ and define
	\[
	d_{\mathbb{H}^2}(g, h) := d_{\mathbb{H}^2}(g.x_0, h.x_0).
	\]
	Due to the Milnor-\v{S}varc lemma, the distance $d_{\mathbb{H}^2}$ is a well-defined distance which is quasi-isometric to the word distance. Let us call such distances \textbf{geometric}.
	
	Let $(G, d)$ be a finitely generated metric group with a left-invariant distance $d$. Consider a nearest-neighbour random walk $(X_i)$ defined by a probability measure $\mu$ with support in the generating set of $G$. Then we can define the following invariants: \textbf{Avez entropy} $h$, \textbf{drift} $l$ and \textbf{logarithmic volume} $v$:
	\[
	\begin{aligned}
	v_d &:= \lim_{n \rightarrow \infty} \dfrac{\log|B_n|}{n} &\quad  &\text{(logarithmic volume)} \\
	h_{\mu} &:= \lim_{n \rightarrow \infty} \dfrac{-\mathbb{E} [\log \mu_n]}{n} &\quad &\text{(Avez entropy)} \\
	l_{d, \mu} &:= \lim_{n \rightarrow \infty} \dfrac{\mathbb{E} [d(e, X_n)]}{n} &\quad &\text{(drift)}, \\
	\end{aligned}
	\]
	where $B_n = \{ g \in G : d(e, g) \le n \}$. 
	
	If these invariants are well-defined, they alone can provide a lot of information about a random walk on a group. In particular, $h = 0$ if and only if the Poisson boundary of the random walk is trivial (see \cite{kaimanovich1983random}, \cite{kaimanovich2000poisson}). Moreover, they are related via the \textbf{fundamental inequality} (for proofs see \cite{guivarc1980loi}, \cite{vershik2000}, \cite{blachere2008asymptotic}):
	\begin{equation}
	\label{introduct fundamental inequality}
	h_{\mu} \le l_{d, \mu} v_d.
	\end{equation}
	There is a well-known problem, which was considered by Y. Guivarc'h, V. Kaimanovich, S. Lalley, A. Vershik, S. Gou\"ezel, and many others (see \cite{guivarc1980loi}, \cite{zbMATH04201213}, \cite{gouezel2014local}, \cite{vershik2000}, \cite{leprince2008}, \cite{kaimanovich2011matrix}, \cite{blachere2011harmonic}, \cite{gouezel2018entropy} for example):
	
	\textbf{Question 1:} how can one classify metric groups and random walks on them for which we have $$h_{\mu} < l_{d, \mu} v_d?$$
	
	In this paper we prove the following theorems:
	\begin{theorem}
		\label{my theorem}
		Let us endow $\Gamma_{n,m}$ with the geometric distance $d = d_{\mathbb{H}^2}$. Then for all but finitely many pairs $(n,m)$ with $n > 3$ and even $m > 3$ we have 
		\[
		h_{\mu} < l_{d, \mu} v_d
		\]
		for the simple random walk on $\Gamma_{n,m}$, i.e., when $\mu$ is the uniform measure on the set of reflections through the sides of $\Delta_{n,m}$.
	\end{theorem}
	Morevoer, in Section \ref{Even case} we show that Theorem \ref{my theorem} holds for even $n \ge 4$ and for all \textbf{geometrically symmetric} nearest-neighbour random walks, that is, such that $\mu(r_i) = \mu(r_i+\frac{n}{2}) > 0$ for every $1 \le i \le \frac{n}{2}$.
	\begin{theorem}
	    \label{my Fuchsian theorem}
	    Let $n \ge 4$ be even, and let $m \ge 3$. Consider a Fuchsian group $\mathcal{F}_{n,m}$ generated by side-pairing translations $(t_i)_{1\le i \le n}$ associated to the polygon $\Delta_{n,m}$, identifying the opposite sides of the polygon. Then for all but finitely many pairs $(n,m)$ we have
	    \[
	    h_{\mu} < l_{d, \mu} v_d
	    \]
	    for any generating \textbf{symmetric} nearest-neighbour random walk on $\mathcal{F}_{n,m}$, i.e. the support of $\mu$ is the generating set $\{ t_i \}_{1 \le i \le n}$, and $\mu(t_i) = \mu(t^{-1}_i) > 0$ for all $1 \le i \le n$. 
	\end{theorem}
	
	\noindent
	\textbf{Remark.} The exceptions for Theorem \ref{my theorem} are the pairs $(n,m) = (4,6),(5,4),(6,4)$, and for Theorem \ref{my Fuchsian theorem} the exceptions are  $(n,m)=(4,5),(4,6),(4,7),(6,4),(8,3),(10,3)$. Notice that $\Delta_{4,4}$ is not a well-defined hyperbolic polygon and does not generate a hyperbolic tiling.
	
	\medskip
	Importance of Question 1 is demonstrated by a connection with another problem related to the behavior of random walks at infinity. Let $(\Gamma, \Sigma)$ be a countable group of isometries of $\mathbb{H}^2$ with a finite generating set $\Sigma= \Sigma^{-1}$. And now let us consider a random walk $X_n$, starting from $e$, defined by a \textbf{generating} probability measure $\mu$ on $\Sigma$, i.e., such that the semigroup generated by the support of $\mu$ equals $\Gamma$.
		
	Recall that almost every sample path of the random walk $(X_n)$ converges to an element of the Gromov boundary $\partial \Gamma$, which is homeomorphic to $S^1$. First results of such kind were discovered by Furstenberg (see \cite{furstenberg1963noncommuting}, \cite{furstenberg71}), some of the more recent results are obtained in \cite{PSMIR_1994___2_A4_0}, \cite{kaimanovich2000poisson}, and \cite{maher2018random}. Therefore, $(X_n)$ induces a measure $\mu_\infty$ on $\partial \mathbb{H}^2 = S^1$, which is called the \textit{hitting measure}. This measure is equivalent to the harmonic measure on the Poisson boundary of $\Gamma$ due to \cite[Theorem 7.6, Theorem 7.7]{kaimanovich2000poisson}. So, one can ask this question.
	
	\medskip
	\textbf{Question 2:} is the hitting measure equivalent to the Lebesgue measure?
	
	As it turns out, there is a strong connection between Question 1 and Question 2. It is illustrated by the results proven in \cite[Corollary 1.4, Theorem 1.5]{blachere2011harmonic} and in \cite{tanaka2017dimension}. For the convenience of the reader we will summarize the results in the following theorem.
	\begin{theorem}[{\cite[Corollary 1.4, Theorem 1.5]{blachere2011harmonic}, \cite{tanaka2017dimension}}]
		\label{main tool}
		Let $\Gamma$ be a non-elementary hyperbolic group acting geometrically on $\mathbb{H}^2$, endowed with the geometric distance $d = d_{\mathbb{H}^2}$ induced from the action of $\Gamma$. Consider a generating probability measure $\mu$ on $\Gamma$ with finite support. Let us also assume that $\mu$ is symmetric. Then the following conditions are equivalent:
		\begin{enumerate}[label=(\arabic*)]
			\item The equality $h_{\mu} = l_{d, \mu} v_d$ holds.
			\item The Hausdorff dimension of the exit measure $\mu_\infty$ on $S^1$ is equal to $1$.
			\item The measure $\mu_\infty$ is equivalent to the Lebesgue measure on $S^1$.
			\item There exists a constant $C > 0$ such that for any $g \in \Gamma$ we have
			\[
			| v_{d} d(e, g) - d_\mu(e, g) | \le C,
			\]
			where $d_\mu(e,g) = - \log(F_\mu(e,g))$ denotes the \textbf{Green metric} associated to $\mu$.
		\end{enumerate}
	\end{theorem}

	One might consider this theorem as a powerful method which can be used to tackle Question 1 and Question 2 at the same time.

	In the case when the distance $d$ is the word metric, the authors of \cite{gouezel2018entropy} used \cite[Corollary 1.4, Theorem 1.5]{blachere2011harmonic} along with an elegant cocycle argument to get the following result:
	
	\begin{theorem}{\cite[Theorem 1.3]{gouezel2018entropy}}
		\label{word metric}
		Let $(\Gamma, \Sigma)$ be a non-elementary non-virtually free hyperbolic group equipped with a generating measure $\mu$. Then, for $d = d_w$, the word metric, we have
		\[
		h_{\mu} < l_{d, \mu} v_d.
		\]
	\end{theorem}
	\noindent
	\textbf{Remark.} It is worth noting that the cohomological machinery which is used to prove this theorem heavily relies on the fact that $d_w$ is an integer-valued distance.
	
	It is a well-known fact (see \cite[Theorem 4.2]{vershik2000}) that for simple random walks on free groups $F_n$ we have $h = lv$, so we have to require the group to be non-virtually free. This is a very powerful result because a lot of interesting non-elementary hyperbolic groups are not virtually free. However, Question 1 is still open for geometric distances induced from geometric actions on $\mathbb{H}^2$. In the non-cocompact case it is known that $h < lv$, see \cite{zbMATH04201213}, \cite{2018arXiv181110849D} or \cite{2019arXiv190411581R}.
	
	\subsection{Our approach}
	
	In this paper we attempt to solve Question 1 for $\Gamma_{n,m}$. Firstly, we prove that $v_{d_{\mathbb{H}^2}} = 1$. For simplicity, let's assume for the moment that we consider the simple random walk on $\Gamma_{n,m}$. The idea is to find a hyperbolic element $g \in \Gamma$ and a point $x_0 \in \mathbb{H}^2$ such that
	\begin{itemize}
		\item $d_{\mathbb{H}^2}(e, g^k) = k d_{\mathbb{H}^2}(e, g)$,
		\item $k d_{\mathbb{H}^2}(e, g) > k |g| \log(|\Sigma|) \ge d_{\mu}(e, g^k)$.
	\end{itemize}
	\noindent
	Then the implication $(4) \Rightarrow (1)$ in Theorem \ref{main tool} implies that $h < lv$.
	
	In the case when $\Gamma = \Gamma_{n,m}$ we can take $\Sigma = \{ r_i \}$, and
	\begin{itemize}
		\item the translation $g = r_1 r_{\frac{n}{2}+1}$ in the case when $n > 3$ is even
		\item the translation $g = r_1 r_{\frac{n+1}{2}}$ in the case when $n > 3$ is odd.
	\end{itemize}
	and compute $d_{\mathbb{H}^2}(e, g)$ explicitly, as shown in Propositions \ref{even case} and \ref{odd case}. This gives us a proof of Theorem \ref{my theorem}.
	
	\noindent
	\textbf{Remark.} It is easily seen that if there exists a point $x_0$ such that $h = lv$ for $d_{\mathbb{H}^2}$, then $h = lv$ for \textit{every} choice of $x_0 \in X$ due to the triangle inequality and Theorem \ref{main tool}(4). Also, keep in mind that this approach will not work for $n=3$, because all sides of a triangle are adjacent to each other.

    In Section \ref{Fuchsian groups} we apply our methods to some Fuchsian groups associated with $\Delta_{n,m}$, as well, thus proving Theorem \ref{my Fuchsian theorem}.
	
	\subsection*{Acknowledgments}
	The author is immensely grateful to Giulio Tiozzo for helpful discussions and bringing this problem to his attention. Also we would like to thank Sebastien Gou\"ezel and Vadim Kaimanovich for valuable comments and suggestions which helped to improve this paper.

	\section{Definitions}
	\subsection{Hyperbolic groups}
	\begin{definition}
		\label{hyperbolic space}
		A geodesic metric space $(X, d)$ is called a \textbf{hyperbolic space} if there exists $\delta> 0$ such that for any geodesic triangle $[x,y]\cup[y,z]\cup[z,x] := \Delta(x,y,z)$ and for any $p \in [x,y]$ there exists $q \in [y,z]\cup[z,x]$ so that $d(p,q) < \delta$. 
	\end{definition}
	
	\begin{definition}
		\label{hyperbolic group}
		Let $G$ be a finitely generated group. TFAE:
		\begin{enumerate}
			\item The Cayley graph $(\Gamma(G, S), d_w)$ is hyperbolic for some generating set $S$
			\item The Cayley graph $(\Gamma(G, S), d_w)$ is hyperbolic for \textit{every} generating set $S$
		\end{enumerate}
		If at least one property holds for $G$, then $G$ is called a \textbf{word-hyperbolic group}.
	\end{definition}

	\begin{definition}
		Finite groups and virtually cyclic groups are called \textbf{elementary} hyperbolic groups.
	\end{definition}
	
	\begin{definition}
		An (isometric) action of a group $G$ on a metric space $X$ is
		\begin{enumerate}
			\item \textbf{properly discontinuous}, if for any compact $K \subset X$ the set
			\[
			\{g \in G \ | \ g K \cap K \neq \emptyset\} \text{ is finite.}
			\]
			\item \textbf{cocompact}, if $X / G$ is compact.
			\item \textbf{geometric}, if it is properly discontinuous and cocompact. 
		\end{enumerate}
	\end{definition}
	
	Recall the well-known Milnor-\v{S}varc lemma:
	\begin{lemma}[Milnor-\v{S}varc lemma]
		\label{Milnor-Svarc}
		A finitely generated group $G$ is word-hyperbolic if and only if $G$ admits a geometric action on a proper hyperbolic metric space $(X, d)$. Moreover, the orbit map
		\[
		t_x : (G, d_w) \rightarrow X, \quad t_x(g) = g.x,
		\]
		is a quasi-isometry.
	\end{lemma}
	For example, any finitely generated group which admits a geometric action on the hyperbolic space $\mathbb{H}^n$ for $n \ge 2$ is a word-hyperbolic group.
	
	\subsection{Random walks and the Green metric}
	\begin{definition}
		Let $(G, S)$ be a finitely generated group. A \textbf{random walk} on $G$ is an infinite sequence of $G$-valued random variables of form
		\[
		X_n = X_0 \xi_1 \dots \xi_n,
		\]
		where $\xi_i$ are i.i.d. $G$-valued random variables, and $X_0$ (initial distribution) is independent from $\xi_i$. If $\xi_i$ take values in $S$ then we say that $(X_n)$ is a \textbf{nearest-neighbor random walk}. If, in addition, $\xi_1$ is uniformly distributed then we will call $(X_n)$ a \textbf{simple} random walk.
	\end{definition}
	\noindent
	\textbf{Remark.} In this paper we only consider nearest-neighbor random walks which start at $e \in G$. Such random walks are uniquely defined by a probability measure $\mu$ on $S$.
	
	Denote the distribution of $X_0$ and $\xi$ by $\mu_0$ and $\mu$, respectively. Then the distribution of $X_n$ is denoted by $\mu_n$. Also, define the first-entrance function $F_{\mu}(x, y)$ as follows:
	\[
	F_{\mu}(x, y) := \mathbb{P}^x(\exists n : X_n = y) = \mathbb{P}^e(\exists n : X_n = x^{-1}y).
	\]
	This also allows us to define the \textbf{Green metric} as follows:
	\[
	d_{\mu}(x, y) := -\log(F_{\mu}(x, y)) \quad \text{for all } x,y \in G.
	\]
	Observe that if $g = s_1 \dots s_{k}$ is a minimal representation of $g$, so that $k = |g|$, where $|g| := d_w(e, g)$ denotes the distance from $e$ to $g$ with respect to the word metric. Therefore,
	
	\begin{equation}
	\label{general naive estimate}
	\mu(s_1) \dots \mu(s_k) \le F_{\mu}(e, g).
	\end{equation}
	
	In particular, for the uniform measure $\mu$ we get

	\begin{equation}
	\label{naive estimate}
		|S|^{-|g|} \le F_{\mu}(e, g),
	\end{equation}
	The proof of \eqref{general naive estimate} is extremely short:
	\begin{equation}
		\label{stronger}
		\mu(s_1) \dots \mu(s_k) = \mathbb{P}^e(\xi_1 = s_1, \dots, \xi_{|g|} = s_{|g|}) \le F_{\mu}(e, g).
	\end{equation}
	
	\section{The main results}
	\subsection{Reflection groups}
	This lemma is a basic and well-known result related to the hyperbolic circle problem.
	\begin{lemma}
		\label{logarithmic volume of gammas}
		For any $n, m$ the logarithmic volume $v$ of $(\Gamma_{n,m}, d_{\mathbb{H}^2})$ equals $1$.
	\end{lemma}
	\begin{proof}
		Denote
		\[
		\# B_R =  |\{ g \in \Gamma_{n,m} : d(x_0, g.x_0) \le R \}|.
		\]
		Let $D_R$ denote the union of the polygons which intersect the closed hyperbolic disk  $B_{\mathbb{H}^2}(x_0,R)$. If we denote the diameter of $\Delta_{n,m}$ by $A$, then
		\[
		B_{\mathbb{H}^2}(x_0,R-A) \subset D_R \subset B_{\mathbb{H}^2}(x_0,R+A),
		\]
		and
		\begin{small}
			\[
			\frac{4\pi \sinh^2((R-A)/2)}{\text{Area}(\Delta_{n,m})} = \frac{\text{Area}(B_{\mathbb{H}^2}(x_0,R-A))}{\text{Area}(\Delta_{n,m})} \le \# B_R \le \frac{\text{Area}(B_{\mathbb{H}^2}(x_0,R+A))}{\text{Area}(\Delta_{n,m})} = \frac{4\pi \sinh^2((R+A)/2)}{\text{Area}(\Delta_{n,m})}.
			\]
		\end{small}

		It is easily seen that $4\pi \sinh^2\left(\frac{R \pm A}{2} \right) \sim e^R$, which immediately yields $v = 1$.		
	\end{proof}

	\begin{lemma}
		\label{main lemma}
		Consider a nearest-neighbour random walk $(X_n)$ on $\Gamma_{n,m}$ such that $\mu(r_i) > 0$ for all $1 \leq i \leq n$. Let $g \in \Gamma_{n,m}$ be a hyperbolic element, in other words, there exists a line $\xi \subset \mathbb{H}^2$ such that $d_{\mathbb{H}^2}(x, g.x) = L > 0$ for all $x \in \xi$. If $g = r_1 \dots r_{|g|}$ and
		\begin{equation}
		\label{the thing we need to check}
			L > - \sum_{i=1}^{|g|} \log(\mu(r_i)),
		\end{equation}
		then
		\begin{equation}
			\label{infinite supremum}
			\sup_{k \rightarrow \infty} | d_\mu(e, g^k) - d_{\mathbb{H}^2}(e, g^k) | = \infty.
		\end{equation}
		for \textbf{any} $x_0 \in \mathbb{H}^2$.
	\end{lemma}
	\begin{proof}
		Choose a point $x_0 \in \xi$, so that
			\[
			d_{\mathbb{H}^2}(e, g^k) = d_{\mathbb{H}^2}(x_0, g^k.x_0) = k L.
			\]
		Due to \eqref{stronger} we obtain
			\[
			d_{\mathbb{H}^2}(e, g^k) - d_\mu(e, g^k) > k L - k \sum_{i=1}^{|g|} \log\left(\frac{1}{\mu(r_i)} \right) = k  \left(L - \sum_{i=1}^{|g|} \log\left(\frac{1}{\mu(r_i)} \right) \right).
			\]
			Due to \eqref{the thing we need to check}, the value $k  \left(L - \sum_{i=1}^{|g|} \log\left(\frac{1}{\mu(r_i)} \right) \right)$ goes to infinity when $k \rightarrow \infty$.
			
			We finish the argument by observing that the choice of $x_0$ doesn't matter due to the triangle inequality.
	\end{proof}

	Now we are going to prove Theorem \ref{my theorem} be considering the cases of even and odd $n$ separately.

	\subsubsection{Even case}
	\label{Even case}
	\begin{theorem}
		Consider the simple random walk $(X_n)$ on $\Gamma_{n,m}$ for even $m, n \ge 4$. If
			\begin{equation}
				4 \text{arccosh} \, \left( \dfrac{\cos(\pi /m)}{\sin(\pi/n)} \right) > 2 \log(n)
			\end{equation}
		then
		\[
		h < lv.
		\]
	\end{theorem}
	\begin{proof}
		Let us define 
		\[
		g = r_1 r_{\frac{n}{2}+1}.
		\] 
		It is, indeed, a translation, and the vertical line $x = 0$ in the Poincar\'e disk model is precisely the axis of $g$. Then we observe that $L = d_{\mathbb{H}^2}(0, g.0) = 4 h_{n,m}$, where $h_{n,m}$ is the altitude of the hyperbolic triangle with angles $\frac{2\pi}{n}$ and $\frac{\pi}{m}, \frac{\pi}{m}$ through $0$.

		The hyperbolic law of cosines shows that 
		\begin{equation}
		\label{altitude}
		h_{n,m} = \text{arccosh}\left( \dfrac{\cos(\pi/m)}{\sin(\pi/n)} \right).
		\end{equation}

		Because $|g| = 2$, the inequality $L > - \sum_{i=1}^{|g|} \log(\mu(r_i))$ can be rewritten as
		\[
		4 \text{arccosh}\left( \dfrac{\cos(\pi/m)}{\sin(\pi/n)} \right) > 2 \log(n),
		\]
		and we can apply Lemma \ref{logarithmic volume of gammas}, Lemma \ref{main lemma} and Theorem \ref{main tool}, keeping in mind that $\Gamma_{n,m}$ is always a non-elementary hyperbolic group.
	\end{proof}
	\noindent
	\textbf{Remark.} Keep in mind that the argument works for any nearest-neighbour random walk generated by such $\mu$ that
	\[
	2 \log(n) > -\log(\mu(r_k)) - \log( \mu(r_{k + \frac{n}{2}}))
	\]
	for some $k \in \mathbb{N}$. In particular, if $\mu(r_i) = \mu(r_{i+\frac{n}{2}})$ for all $i \in \mathbb{N}$, then we can always find such $k$.

	\begin{proposition}
	\label{even case}
		The inequality 
		\begin{equation}
		\label{arccos}
		4 \text{arccosh} \, \left( \dfrac{\cos(\pi /m)}{\sin(\pi/n)} \right) > 2 \log(n)
		\end{equation}
		holds for 
		\begin{itemize}
		    \item $n \ge 4, m \ge 8$,
		    \item $n \ge 6, m \ge 5$,
		    \item $n \ge 8, m \ge 4$,
		    \item $n \ge 12, m \ge 3$.
		\end{itemize}
	\end{proposition}
	\noindent
	\textbf{Remark.} The exact region where the inequality \eqref{even inequality} holds is illustrated by the Figure \ref{fig1}.
	\begin{figure}[h]
	    \centering
	    \includegraphics[width=10cm]{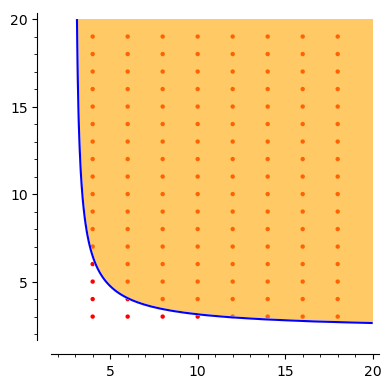}
	    \caption{The dots in the orange region correspond to the pairs $(n,m)$ for which \eqref{even inequality} holds. Keep in mind that we still require both $n > 3$ and $m > 3$ to be even, so the exceptional cases here are $(4,4),(4,6),(6,4)$.}
	    \label{fig1}
	\end{figure}
	
	\begin{proof}
		By definition, $\text{arccosh}(x) = \ln(x + \sqrt{x^2-1})$, so \eqref{arccos} is equivalent to
		\begin{equation}
		\label{even inequality}
		\left(\dfrac{\cos(\pi /m)}{\sin(\pi/n)} + \sqrt{\dfrac{\cos(\pi /m)^2}{\sin(\pi/n)^2} - 1}\right)^2 > n.
		\end{equation}
		For convenience, let us denote $f(n, m) = \dfrac{\cos(\pi /m)}{\sin(\pi/n)}$. The following lemma is a straightforward corollary from $\eqref{altitude}$.
		\begin{lemma}
			\indent
			\begin{enumerate}[label=(\arabic*)]
				\item $f(n, m)$ is a separately strictly increasing function. In other words,
				\[
				f(n, m) < f(n+1, m), \quad f(n, m) < f(n,m+1) \quad (m,n \ge 3).
				\]
				\item $f(n, m) \ge 1.25$ for $n \ge 4, m \ge 7$; $n \ge 6, m \ge 4$; and $n \ge 8, m \ge 3$.
			\end{enumerate}
		\end{lemma}
		\begin{proof}[Proof of the lemma.]
			\indent
			\begin{enumerate}[label=(\arabic*)]
				\item This immediately follows from the monotonicity of $\cos$ and $\sin$ on $[0, \frac{\pi}{2}]$.
				\item Notice that
				\[
				\begin{aligned}
				f(4, 7) &\approx 1.27416, \\
				f(6, 4) &\approx 1.41421, \\
				f(8, 3) &\approx 1.30656,
				\end{aligned}
				\]
				so we can use (1) to get the inequality for the remaining cases via monotonicity.
			\end{enumerate}
		\end{proof}
		First of all, recall that
		\begin{equation}
		\label{1.25 square root estimate}
		\sqrt{x^2-1} \ge x-\frac{1}{2} \text{ for all } x \ge \frac{5}{4}.
		\end{equation}
		Therefore, we can apply our simple lemma and \eqref{1.25 square root estimate} to get
		\begin{equation}
		(f(n, m) + \sqrt{f(n,m)^2-1})^2 \ge \left(2f(n, m) - \frac{1}{2} \right)^2.
		\end{equation}
		So, instead of checking \eqref{even inequality}, let us check a slightly stronger inequality:
		\begin{equation}
		\left(2f(n, m) - \frac{1}{2} \right)^2 > n, 
		\end{equation}
		which is equivalent to
		\begin{equation}
		\label{stronger even inequality}
		\dfrac{\cos(\frac{\pi}{m})}{\sin(\frac{\pi}{n})} > \dfrac{\sqrt{n} + \frac{1}{2}}{2}.
		\end{equation}
		If we multiply both sides by $\sin\left(\frac{\pi}{n} \right)$ and take $\arccos$, we get
		\[
		f_l(m) := \cos\left(\frac{\pi}{m} \right) > \sin\left(\frac{\pi}{n}\right) \dfrac{\sqrt{n} + \frac{1}{2}}{2}  =f_r(n).
		\]
		Thankfully, $f_l$ is strictly increasing for $m \ge 4$, and $f_r$ strictly decreasing for all $n \ge 4$. In particular,
		\[
		\begin{aligned}
		f_l(8) &\approx 0.923879, & f_r(4) &\approx 0.88388, \\
		f_l(6) &\approx 0.866, & f_r(6) &\approx 0.73737, \\
		f_l(4) &\approx 0.7071, & f_r(8) &\approx 0.63686, \\
		f_l(3) &= 0.5, & f_r(14) &\approx 0.4719.
		\end{aligned}
		\]
		In particular, this shows that there are only finitely many cases for which \eqref{even inequality} doesn't hold. The remaining cases $(n,m) = (6,5)$ and $(n,m) = (12,3)$ can be verified separately.
	\end{proof}

	\subsubsection{Odd case}
	\begin{theorem}
		Let $(X_n)$ denote the simple random walk on $\Gamma_{n,m}$ where $n \ge 5$ is odd and $m \ge 4$ is even. If
		\begin{equation}
			\label{odd inequality}
			 \sin(\pi / m) \cosh \left( \text{arccosh}\left(\frac{\cos(\pi / m)}{\sin(\pi / n)} \right) + \text{arccosh}(\cot(\pi/m) \cot(\pi/n)) \right) > \cosh(\log(n))
		\end{equation}
		then
		\[
		h < lv.
		\]
	\end{theorem}
	\begin{proof}
		WLOG we can assume that $k = 1$ and we can define 
		$$
		g = r_1 r_{\frac{n+1}{2}}.
		$$ 
		Finding the translation length of $g$ is slightly less trivial in the odd case, because the respective sides of $\Delta_{n,m}$ are not opposite to each other. However, let us consider a hyperbolic line which is orthogonal to the sides corresponding to $r_1$ and $r_{\frac{n+1}{2}}$. Thus, $L_g$ equals to doubled distance between the points where this line intersects $\Delta_{n,m}$, and we can compute it by noticing that it is a side of a Lambert quadrilateral. Therefore,
		\[
		L_g = 2 \text{arccosh} \, \left(\sin \left(\frac{\pi}{m} \right) \cosh(a_{n,m}) \right),
		\]
		where
		\[
		\begin{aligned}
		a_{n,m} &= \text{arccosh}\left(\frac{\cos(\pi / m)}{\sin(\pi / n)} \right) + \text{arccosh}\left(\frac{\cos(\pi / m) + \cos(\pi/m)\cos(2 \pi / n)}{ \sin(\pi/m) \sin(2\pi / n)} \right) = \\ &= \underbrace{\text{arccosh}\left(\frac{\cos(\pi / m)}{\sin(\pi / n)} \right)}_{\text{the length of the altitude}} + \underbrace{\text{arccosh}(\cot(\pi/m) \cot(\pi/n))}_{\text{the distance from center to a vertex}},
		\end{aligned}
		\]
		because
		\[
		\dfrac{1 + \cos(2x)}{\sin(2x)} = \dfrac{2\cos^2(x)}{2 \sin(x) \cos(x)} = \cot(x).
		\]
		Therefore, the inequality $L > - \sum_{i=1}^{|g|} \log(\mu(r_i))$ can be rewritten as
		\[
		L_g = 2 \text{arccosh} \, \left(\sin \left(\frac{\pi}{m} \right) \cosh(a_{n,m}) \right) > 2\log(n),
		\]
		which is equivalent to
		\[
		\sin \left(\frac{\pi}{m} \right) \cosh(a_{n,m}) > \cosh( \log(n) ).
		\]
		We finish the argument by applying Lemma \ref{logarithmic volume of gammas}, Lemma \ref{main lemma} and Theorem \ref{main tool}.
	\end{proof}
	\noindent
	\textbf{Remark.} Keep in mind that the argument works for any nearest-neighbour random walk generated by such $\mu$ that
	\[
		2 \log(n) \ge -\log(\mu(r_k)) - \log( \mu(r_{k + \frac{n \pm 1}{2}}))
	\]
	for some $k \in \mathbb{N}$.
	
	\begin{proposition}
	\label{odd case}
		The inequality \eqref{odd inequality} holds for 
		\begin{itemize}
		    \item $n \ge 5, m \ge 6$,
		    \item $n \ge 7, m \ge 4$.
		\end{itemize}{}
	\end{proposition}
	
	\noindent
	\textbf{Remark.} The exact region where the inequality \eqref{odd inequality} holds is illustrated by the Figure \ref{fig2}.
	\begin{figure}[h]
		\centering
		\includegraphics[width=10cm]{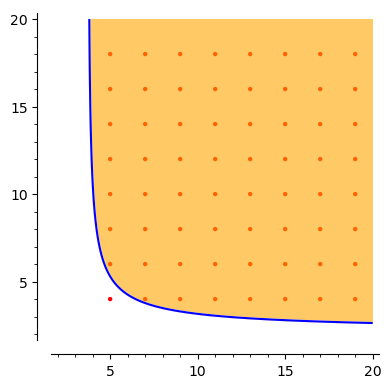}
		\caption{The dots in the orange region correspond to the pairs $(n,m)$ for which \eqref{odd inequality} holds. The only exceptional pair in this case is $(n,m) = (5,4)$.}
		\label{fig2}
	\end{figure}
		
	\begin{proof}
		Equivalently, we want to prove that
		\begin{equation}
		\label{simplified odd inequality}
		\sin \left(\frac{\pi}{m} \right) \cosh(a_{n,m}) + \sqrt{ \sin \left(\frac{\pi}{m} \right)^2 \cosh(a_{n,m})^2 - 1 } \ge n.
		\end{equation}
		\begin{lemma}
		    \indent
			\begin{enumerate}[label=(\arabic*)]
				\item Denote $g(n, m) = \sin \left(\frac{\pi}{m} \right) \cosh(a_{n,m})$. Then $g(n, m)$ is a separately strictly increasing function for $m,n \ge 3$.
				\item $g(n, m) > 1$. 
			\end{enumerate}
		\end{lemma}
		\begin{proof}
			\begin{enumerate}[label=(\arabic*)]
				\item Let's rewrite $g(n, m)$ using the trigonometric identity
				\[
				\cosh(x+y) = \cosh(x)\cosh(y) + \sinh(x)\sinh(y).
				\]
				We get
				\begin{equation}
				\label{def}
				\begin{aligned}
				g(n,m) = & (\cos(\pi / m)\cot(\pi / n)) \sqrt{\frac{\cos(\pi / m)^2}{\sin(\pi / n)^2} - 1} + \\ &+ \frac{\cos(\pi / m)}{\sin(\pi / n)} \sqrt{\left( \cos(\pi / m)\cot(\pi / n) \right)^2 - \sin(\pi/m)^2}.
				\end{aligned}
				\end{equation}
				
				One easily can check that every term in this expression is monotone in $n$ and $m$.
				
				\item This readily follows from the fact that $L_g$ is a well-defined positive number.
			\end{enumerate}
		\end{proof}
		We can use the estimate
		\begin{equation}
		\label{a classical square root estimate}
		\sqrt{x^2-a^2} \ge x-a \quad (x \ge a > 0)
		\end{equation}
		to obtain
		\begin{equation}
		g(n,m) + \sqrt{g(n,m)^2-1} \ge 2g(n, m) - 1 > n.
		\end{equation}
		This is equivalent to
		\begin{equation}
		\label{slightly stronger odd inequality}
		g(n,m) > \dfrac{n+1}{2},
		\end{equation}
		so now we want to solve this inequality which is a bit stronger than \eqref{simplified odd inequality}. However, we can expand \eqref{def} even further using \eqref{a classical square root estimate}:
		Let us expand $g(n, m)$:
		\begin{equation}
		\begin{aligned}
		g(n,m) >  & \left(\frac{\cos(\pi / m)}{\sin(\pi / n)} - 0.5\right) \cos(\pi / m) \cot(\pi / n) + \\ &+ \frac{\cos(\pi / m)}{\sin(\pi / n)} \left(\cos(\pi / m)\cot(\pi / n)  - \sin(\pi/m) \right).
		\end{aligned}
		\end{equation}
		Multiply both parts by $\dfrac{\sin^2(\pi/n)}{\cos(\pi/n)}$:
		\begin{equation}
		2\cos^2\left(\pi/m\right) - 0.5 \sin( \pi/n) \cos(\pi/m) - 0.5 \sin(2\pi/m) \tan(\pi/n) > \dfrac{\sin^2(\pi/n)}{\cos(\pi/n)} \dfrac{n+1}{2}.
		\end{equation}
		We can use $\cos(\pi/m) < 1$ to get
		\begin{equation}
		g_L(n, m) := 2\cos^2(\pi/m) - 0.5 \sin(\pi/n) - 0.5 \sin(2\pi/m) \tan(\pi/n) > \dfrac{\sin^2(\pi/n)}{\cos(\pi/n)} \dfrac{n+1}{2} =: g_R(n).
		\end{equation}
		This inequality is particularly nice because $g_L$ is a separately strictly increasing function and $g_M$ is stricly decreasing. Therefore, it is enough to check the inequality for several particular values of $n,m$.
		
		Another check in Wolfram Mathematica shows that this inequality holds for
		\[
		\begin{aligned}
		g_L(5,10) &\approx 1.3016 & g_R(5) &\approx 1.28115 \\
		g_L(7,5) &\approx 0.863073 & g_R(7) &\approx 0.83579 \\
		g_L(9,4) &\approx 0.647005 & g_R(9) &\approx 0.622426.
		\end{aligned}
		\]
		The remaining cases $(n,m) = (5,6),(5,8),(7,4)$ can be checked manually.
	\end{proof}

	\subsection{Fuchsian groups}
	\label{Fuchsian groups}
	Various studies on the hyperbolic circle problem (see \cite{PHILLIPS199478}, for example) show that $v = 1$ for cocompact Fuchsian groups. Moreover, it is a well-known fact that any cocompact Fuchsian group $\Gamma$ admits a Dirichlet domain $\Delta_{\Gamma}$, which is an even-sided hyperbolic polygon, where each side corresponds to a generator of $\Gamma$, and the resulting system will be minimal (see \cite{book:291162}). Let us denote this symmetric system by $\Sigma$.
	
	Recall that a random walk $(X_n)$ on a group $\Gamma$ defined be a probability measure $\mu$ is \textbf{symmetric} if $\mu(s) = \mu(s^{-1})$ for all $s \in \Sigma$. A simple modification of Lemma \ref{main lemma} allows us to formulate the following result.
	\begin{theorem}
	    \label{Fuchsian theorem}
		Let $(\Gamma, \Sigma) \subset \text{PSL}(2, \mathbb{R})$ be a cocompact Fuchsian group with the Dirichlet domain $\Delta_{\Gamma}$, and let us consider a generating symmetric nearest-neighbour random walk $(X_n)$ defined by a probability measure $\mu$. Suppose that there exists a hyperbolic element $g \in \Gamma$ and $x_0 \in \mathbb{H}^2$ such that
			\[
			L_g = d_{\mathbb{H}^2}(x_0, g.x_0) > - \sum_{i=1}^{|\Sigma|} \log(\mu(s_i)). 
			\]
		Then, for $d = d_{\mathbb{H}^2}$ we have
		\[
		h_\mu < l_{d, \mu} v_d = l_{d, \mu}.
		\]
	\end{theorem}

	Now we want to apply this theorem to a concrete family of Fuchsian groups. For an even $n \ge 4$, given a regular hyperbolic polygon $\Delta_{n,m}$, one can define a Fuchsian group $\mathcal{F}_{n,m}$ which is generated by translations $(t_i)_{1 \le i \le n}$ such that the axis of $t_i$ goes through the centers of $\Delta_{n,m}$ and $r_i(\Delta_{n,m})$, where $r_i(\Delta_{n, m})$ is the polygon in the tessellation with shares the $i$-th side with $\Delta_{n, m}$, and $t_i$ takes the center of $\Delta_{n,m}$ to the center of $r_i(\Delta_{n,m})$. It is a cocompact Fuchsian group because every element of $\mathcal{F}_{n,m}$ preserves the hyperbolic tessellation induced by $\Delta_{n,m}$, and the action is transitive on the tiles. Therefore, the fundamental domain will be contained in $\Delta_{n,m}$. In particular, $\mathcal{F}_{n,m}$ is a non-elementary hyperbolic group and Theorems \ref{main tool} and \ref{Fuchsian theorem} apply.
	
	\begin{proof}[Proof of Theorem \ref{my Fuchsian theorem}]
		Suppose that $(n,m)$ satisfy the inequality \eqref{even inequality}, where $n \ge 4$ is even. Consider the regular hyperbolic polygon $\Delta_{n,m}$ and a nearest-neighbour symmetric random walk on $\mathcal{F}_{n,m}$ generated by $\mu$. Since $|\Sigma| = n$, we can always choose such $i$ that $\mu(t_i) \ge \frac{1}{n}$. But it is easily seen that because $L_g = 2h_{n,m}$, the inequality $L_g > -\log(\mu(t_i))$ follows from \eqref{arccos} (or, equivalently, \eqref{even inequality}):
		\[
		L_g = 2h_{n,m} \stackrel{\eqref{arccos}}{>} \log(n) \ge -\log(\mu(t_i)).
		\]
		Therefore, we proved that for \textbf{every} generating symmetric nearest-neighbour random walk on $\mathcal{F}_{n,m}$ we have $h < lv$ with respect to the hyperbolic distance. And by Proposition \ref{even case}, we know that there are only finitely many exceptional cases: $(n,m)=(4,5),(4,6),(4,7),(6,4),(8,3),(10,3)$.
	\end{proof}
	Moreover, we claim that this is somewhat a general occurrence for the simple random walks on cocompact Fuchsian groups generated by hyperbolic elements.
	
	Suppose that the diameter $\Delta_{\Gamma}$ equals $2R$, and $2R > \log(|\Sigma|)$. Then, due to the triangle inequality, for any generator $g \in \Gamma$ we have
	\[
	L_g \ge 2R > \log(|\Sigma|) := \log(2n),
	\]
	and we can apply Theorem \ref{Fuchsian theorem}.

	As one is able to see, we just reduced the question to a non-group-theoretic one: we are comparing two purely geometric values. We claim that this is quite likely to happen, because we can assume that if $\Delta_{\Gamma}$ is ``close'' to a regular polygon, then its area can be approximated by the area of a hyperbolic ball $B(x_0, R)$, which, in turn, is approximately $(2n-2)\pi$. So,
	\[
	4 \pi \sinh^2(R/2) \approx 4 \pi e^R \approx (2n-2)\pi \Rightarrow R \approx \log\left(\frac{n-1}{2} \right),
	\]
	and
	\[
	2 \log\left(\frac{n-1}{2}\right) > \log(2n)
	\]
	for $n \gg 1$. This gives us an idea that for Fuchsian groups with a large number of generators we are more likely to have $h < lv$.

	\printbibliography
	\Addresses
\end{document}